\documentclass{article}
\usepackage[a4paper, textwidth=500pt, tmargin=.95in]{geometry}

\makeatletter
\renewcommand\tableofcontents{%
    \@starttoc{toc}%
}
\makeatother

\usepackage{Timstyle}

\usepackage{lettrine}
\usepackage[normalem]{ulem}
\usepackage{authblk}
\DeclareMathAlphabet{\mathcal}{OMS}{cmsy}{m}{n}
\renewcommand{\k}{\mathfrak k}

\title{The Hitchin and Knizhnik--Zamolodchikov connections\\ are projectively equivalent in the genus zero case\footnotemark[1]}
\author[1,2]{J\o rgen Ellegaard Andersen}
\author[3]{Tim Henke}
\affil[1]{Centre for Quantum Mathematics, University of Southern Denmark}
\affil[2]{Danish Institute of Advanced Study} 
\affil[3]{Centre for Mathematics of the University of Porto}

\renewcommand{\git}{/\!\!/} % GIT quotient
\newcommand{\Sigmao}{\Sigma^\circ}

\newcommand{\vmu}{\vec \mu}

\pgfplotsset{compat=1.18}
\begin{document}

\maketitle 

\renewcommand{\thefootnote}{\fnsymbol{footnote}}

\footnotetext[1]{This work is supported by the "ReNewQuantum" ERC Synergy grant
 agreement No. 810573 and the Simons Foundation collaboration grant on New Structures in Low-Dimensional Topology.} \footnotetext[1]{This work is partially supported by FCT (Fundação para a Ciência e Tecnologia), under the projects with reference UID/00144/2025, and associated DOI \url{https://doi.org/10.54499/UID/00144/2025}, CMUP, member of LASI, and 2024.15931.PEX Higgs bundles: geometry, algebra and physics.}

\abstract{
This paper establishes the projective equivalence between the Knizhnik--Zamolodchikov connection and the Hitchin connection in genus 0 with at least 3 marked points. The Knizhnik--Zamolodchikov connection is defined on the sheaf of conformal blocks in the Tsuchiya--Ueno--Yamada model of conformal field theory. The Hitchin connection is defined on the Verlinde bundle via geometric quantisation of the moduli space of flat connections. Pauly's isomorphism establishes the equivalence of these two vector bundles. The main theorem of this paper is that the isomorphism intertwines these two connections up to a scalar-valued one-form. In addition, this theorem is used to construct a Hitchin connection through an auxiliary metaplectic correction. As a corollary of the main theorem, this construction of the Hitchin connection is projectively unique and projectively flat.
}

\section{Introduction}

Classical Chern--Simons theory was developed originally in \cite{chern1974characteristic} and first discussed as a quantum theory in \cite{deser1982three, schonfeld1981mass}, which was further expanded on in \cite{witten1989quantum, axelrod34geometric, ramadas1989some}. The Hitchin connection, originally studied in \cite{axelrod34geometric}, mathematically constructed in \cite{hitchin1990flat}, and researched and generalised further in a large number of subsequent publications including \cite{van1998hitchin, andersen2012hitchin, andersenmetaplectic, laszlo2013monodromy, andersen2014hitchin, ouaras2023parabolic, biswas2023hitchin, biswas2024geometrization}, governs the dynamics of states in quantum Chern--Simons gauge theory under deformations of the Riemann surface on which they depend and establishes the topological nature of quantum Chern--Simons theory. The Knizhnik--Zamolodchikov connection, based on the Knizhnik--Zamolodchikov equations derived in \cite{knizhnik1984current}, determines the Riemann surface dependence of covacuum states of Wess--Zumino--Novikov--Witten conformal field theory, developed in \cite{wess1971consequences, witten1983global} and formalised in \cite{kawamoto1988geometric, TUY, ueno2008conformal}.

Quantum Chern--Simons gauge theory for compact gauge groups and Wess--Zumino--Novikov--Witten conformal field theory were first suggested to be equivalent theories on physical grounds by Witten in \cite{witten1989quantum}, and this equivalence was formalised mathematically in various stages in \cite{beauville1994conformal,kumar1997infinite,pauly1996espaces, laszlo1997line}. In full generality, the theorem states that there is an isomorphism, which shall be referred to as Pauly's isomorphism, which is natural up to projective equivalence, between the geometric quantisation of the moduli space of flat connections over a marked and labelled Riemann surface and the vector space of covacua over the same marked and labelled Riemann surface in the Tsuchiya--Ueno--Yamada model of conformal field theory.

The main theorem of this paper, \Cref{thm:main}, proves that Pauly's isomorphism gives a projective equivalence of the Knizhnik--Zamolodchikov connection on the bundle of covacua over Teichmüller space and the Hitchin connection on the Verlinde bundle, also over Teichmüller space, whose fibres are obtained by applying geometric quantisation to the moduli space of parabolic bundles, for all simple complex Lie groups in genus zero. Unlike in higher genera, the genus zero case cannot be deduced from the non-parabolic counterpart. As a consequence of the theorem, a new construction of the Hitchin connection is developed via an auxiliary metaplectic correction, and this Hitchin connection is shown to be in general projectively unique and projectively flat.

The proof of the main theorem is achieved via the construction of a geometric model of the Knizhnik--Zamolodchikov connection as a differential operator on the GIT quotient of a product of flag varieties via Bott--Borel--Weil theory. This GIT quotient shares a dense Zariski-open subset with the moduli space of parabolic bundles on which the two second-order differential operators can be compared. Via a careful analysis of the scaling behaviour of both differential operators under increasing powers of the line bundle, their difference can be determined to vanish in both the second and the first order. Thus the difference is given by multiplication by a holomorphic function on the dense open subset. Making use of the fact that both operators preserve the sections that extend holomorphically, this function must itself extend holomorphically to the compact spaces and therefore be constant, proving their projective equivalence. While the constructions rely on a substantial amount of pre-existing work, the proof consists purely of elementary arguments.

The results of \cite{TUY, ueno2008conformal}, which include the Knizhnik--Zamolodchikov connection construction in genus zero, was used in \cite{AU1, AU2, AU3, AU4} to develop a topological quantum field theory based on Tsuchiya--Ueno--Yamada conformal field theory that is equivalent to the Witten--Reshetikhin--Turaev (WRT) topological quantum field theory, for which genus zero is decisive. This equivalence makes it possible to perform calculations and proofs concerning the combinatorial WRT topological quantum field theory (TQFT) with geometric methods and arguments involving the moduli spaces of flat connections, such as Andersen's the Asymptotic Faithfulness Theorem \cite{andersen2006asymptotic}.

A closely related result was proven in \cite{laszlo1998hitchin}, showing that the Hitchin connection and the Tsuchiya--Ueno--Yamada connections coincide for genus at least 3 without parabolic points.

In \cite{biswas2021ginzburg} a connection on the Verlinde bundle was introduced in terms of heat kernels, which was identified in \cite{biswas2024geometrization} with the Tsuchiya--Ueno--Yamada connection. The heat kernel is not a priori connected to the original construction of the Hitchin connection in \cite{hitchin1990flat} and its generalisation in \cite{andersen2012hitchin, andersenmetaplectic}. Hence the work done in \cite{biswas2023hitchin, biswas2024geometrization} is at this point in time logically independent of the results in this paper. However, by combining the result in this work with that of \cite{biswas2023hitchin, biswas2024geometrization}, it can therefore be concluded that the connection constructed from the heat kernel is a Hitchin connection in genus zero in the sense of \cite{andersen2012hitchin, andersenmetaplectic}, coinciding with the metaplectic-corrected Hitchin connection. Further important work on the Hitchin and Knizhnik--Zamolodchikov connections in the genus zero case was done in \cite{faltings1993stable, ramadas1998faltings, sun2004hitchin, ran2006jacobi, ramadas2009harder, belkale2009strange, pauly2023hitchin}.

This work has furthermore a clear relation to quantum computing, given the work of Freedman and collaborators \cite{Freedman&Co1,Freedman&Co}, which has established that the WRT-TQFT braid-representation of the genus zero mapping class group at level $k=3, 5, 7, 8, \ldots$ is a theoretical universal quantum computer, when combined with the isomorphism established in \cite{AU1, AU2, AU3, AU4} linking these combinatorial isomorphism to the Knizhnik--Zamolodchikov representations studied in this work. In particular, one gets a topological code, which supports universal quantum computing, yet is efficient, given the low codimension, which the space of conformal blocks has inside the tensor product of the corresponding irreducible representation. Further, its TQFT connection and our symplectic reduction picture developed in this paper, links it to the symplectic reduction view point on quantum low density parity check codes (qLDPC) initiated in \cite{LDq1, LDq2, LDq3} and further developed in a large number of references, see e.g. the review \cite{LDq4} and references therein.

\section{The Knizhnik--Zamolodchikov connection} \label{KZC} 

In this paper, $K$ will be a simple, connected, and simply-connected compact real Lie group with complexification $G$, and $\mathfrak k$ and $\g$ their respective Lie algebras. This section introduces the Knizhnik--Zamolodchikov connection on a trivial bundle over the open configuration space of points in the projective line, whose fibres are the $\g$-invariant subspace of a tensor product of $\g$-representations. We recall how this connection preserves the subbundle consisting of the spaces of conformal blocks.

Choose a Cartan subalgebra ${\mathfrak t} \subset \g$ and a set of positive roots $\Phi_+\subset \Phi$ in the root system of $\g$ relative to $\mathfrak t$. Fix an invariant inner product $\langle\cdot,\cdot\rangle$ on $\g$ normalised so the longest root $\theta$ of $\g$ has norm squared equal $2$. Let $V_\lambda$ be the irreducible $\g$-representation corresponding to a dominant weight $\lambda$ of $\g$. Write $h$ for the dual Coxeter number of $\g$. 

For the remainder of the paper, $k$ will denote a positive integer that will be referred to as the \emph{level}. A weight $\lambda$ is said to be $k$-admissible if $0 \leq \langle\lambda,\theta\rangle \leq k$ and the set of $k$-admissible weights will be written as $\Lambda_k$.

For a tuple of dominant weights $\vlambda = (\lambda_1, \ldots, \lambda_n)$ let $V_\vlambda \defeq V_{\lambda_1} \otimes_\C \cdots \otimes_\C V_{\lambda_n}$. There is a diagonal action of $\g$ on $V_\vlambda$ and we denote by $V_\vlambda^{\g} \subset V_\vlambda$ the invariant subspace under this diagonal action.

Let $C_n(X)$ be the open configuration space of $n$ points on a Riemann surface $X$. Since $\rm{PSL}(2,{\mathbb C})$ acts simply transitive on $C_3(\CP)$, for $n\geq 3$ the moduli space of genus zero smooth curves with $n$ ordered marked points ${\mathbb M}_n$ is
$$ {\mathbb M}_n = C_n(\CP)/\mathrm{PSL}(2,{\mathbb C}) \cong C_{n-3}({\mathbb C}^\times -\{1\}).$$

\begin{definition}\index{Knizhnik--Zamolodchikov connection}
The \emph{Knizhnik--Zamolodchikov connection} is the connection in the trivial $V_\vlambda^\g$-bundle over $ {\mathbb M}_n$ given by
\begin{align} \label{eq:KZ}
\nabla^{\rm KZ} & = \d + \mathbf \Omega &\where & & \mathbf \Omega = \frac{1}{k+h} \sum_{1 \leq i < j \leq n} \Omega^{ij} \frac{\d z_i - \d z_j}{z_i - z_j} 
\end{align}
and $\Omega^{ij}$ is the action of the Casimir element $\Omega \in \g \otimes \g$ on the $i$th and $j$th component of $V_\vlambda$.
\end{definition}

It is well known that the Knizhnik--Zamolodchikov connection is flat. Recall that the Knizhnik--Zamolodchikov connection preserves the subbundle consisting of the Tsuchiya--Ueno--Yamada (TUY) spaces of covacua and induces the TUY-connection on this subbundle of covacua constructed in \cite{TUY} \cite[Proposition 3.5.1]{ueno2008conformal}.

\begin{thm}[{\cite[Equation 2.2-16]{TUY}}]
If $\vlambda$ is an $n$-tuple of $k$-admissible weights, the TUY sheaf of covacua construction applied to ${\mathbb M}_n$ gives a subbundle $ \mathcal V_{\vlambda,k}^\dagger$ of the trivial $V_\vlambda^\g$-bundle over $ {\mathbb M}_n$ which is preserved by $\nabla^{\rm KZ}$ and the induced connection on the subbundle $ \mathcal V_{\vlambda,k}^\dagger$ coincides with the TUY connection $\nabla^{\rm TUY}$.
\end{thm}

\begin{rem}
Throughout the rest of the paper the notation $\mathcal V_{\vlambda,k}^\dagger$ will also be used for the fibre of $\mathcal V_{\vlambda,k}^\dagger$ over any $\vec p\in {\mathbb M}_n$, even if this subspace inside $V_\vlambda^\g$ does depend on $\vec p$. It should be clear from the context which $\vec p$ is being considered.
\end{rem}

We further recall the isomorphism constructed in \cite{AU4} restricted to the genus zero case considered in this paper.

\begin{thm}[{\cite[Theorem 6.10]{AU4}}]
The representation of the mapping class group of a genus zero surface with $n$ marked points and fixed tangent vectors, which arrises form the Knizhnik--Zamolodchikov connection restricted to the subbundle $ \mathcal V_{\vlambda,k}^\dagger$ is equivalent to the combinatorially constructed WRT--TQFT representation for $K=\SU(N)$ at level $k$.
\end{thm}

\section{Geometrisation of the Knizhnik--Zamolodchikov connection}

This section uses Bott--Borel--Weil theory to provide a geometrisation of the Knizhnik--Zamolodchikov connection via the real coadjoint orbit construction.

The real coadjoint orbit $\OO_\lambda \subset \mathfrak k^*$ through $\lambda$ can be identified with the quotient $\OO_\lambda\cong K/Z_\lambda$, where $Z_\lambda\subset K$ is the centraliser of $\lambda$. The Kostant--Kirillov--Souriau symplectic form on $\OO_\lambda$ is given at $\alpha \in \OO_\lambda$ by 
\[
\omega_\alpha(\underline\xi(\alpha),\underline\eta(\alpha)) \defeq \alpha([\xi,\eta])
\]
on the vector fields $\underline \xi,\underline\eta \in \XX(\OO_\lambda)$ generated by the action of Lie algebra elements $\xi,\eta \in \g.$
The character $\chi_\lambda : Z_\lambda \to U(1)$ given by $\exp(2\pi i \xi) \mapsto \exp(2\pi i \langle \xi,\lambda\rangle)$ defines the Hermitian line bundle 
\[
L_\lambda \defeq \C \times_{Z_\lambda} K \vb \OO_\lambda.
\]
There is a unique Hermitian connection $\nabla$ in $L_\lambda$ whose curvature is the Kostant--Kirillov--Souriau symplectic form $\omega$, making it a prequantum line bundle. Recall further that $\lambda$ determines a parabolic subgroup $P_\lambda \subset G$ and a corresponding flag variety $F_\lambda \defeq G/P_\lambda$. 

\begin{lemma}[{\cite[Section 7.4]{huckleberry2012infinite}}]
The natural identification $F_\lambda \cong \OO_\lambda$ induces a complex structure on $\OO_\lambda$ compatible with the symplectic form $\omega$ making it a Kähler manifold.
\end{lemma}

The identification $F_\lambda \cong \OO_\lambda$ will be used freely throughout the paper. The complex structure on $\OO_\lambda$ and the hermitian connection equip $L_\lambda$ with the structure of a holomorphic line bundle.

\begin{thm}[Bott--Borel--Weil {\cite[Theorem V]{bott1957homogeneous}}] \label{thm:bbw}
The space $H^0(F_\lambda,L_\lambda)$ of global holomorphic sections of $L_\lambda$ over $F_\lambda$ as a Lie algebra representation is the highest weight representation $V_\lambda$, with the Lie algebra action $\g \times H^0(F_\lambda,L_\lambda) \to H^0(F_\lambda,L_\lambda)$ given pointwise at $\overline g = gP_\lambda \in G/P_\lambda$ by 
$$(\xi \cdot s)(\overline g) = (\nabla_{\underline{\xi}_{ \overline g}} s)(\overline g) - i\lambda(\Ad_{g\inv}\xi)s(\overline g).$$
for all $s\in H^0(F_\lambda,L_\lambda)$.
\end{thm}

Defining $F_\vlambda \defeq F_{\lambda_1} \times \dots \times F_{\lambda_n}$ and $L_\vlambda \defeq L_{\lambda_1} \boxtimes \cdots \boxtimes L_{\lambda_n}$ it is immediate that
\[
V_\vlambda \cong H^0(F_{\lambda_1},L_{\lambda_1}) \otimes \cdots \otimes H^0(F_{\lambda_n},L_{\lambda_n}) \cong H^0(F_\vlambda,L_\vlambda)
\]
as $\g$-representations and therefore $V_\vlambda^{\g} \cong H^0(F_\vlambda,L_\vlambda)^{\g}$.

We now consider the GIT quotient $\FF_\vlambda \defeq (F_\vlambda \git G)$.
If the line bundle $L_\vlambda \vb F_\vlambda$ descends to a line bundle $\L_\vlambda$ over $\FF_\vlambda$, then by the Kempf--Ness Theorem, the GIT quotient of the cross product of the coadjoint orbits can be identified with the symplectic reduction $F_\vlambda\git G \cong \OO_\vlambda {\symp}_{\!\mu} K$ along the canonical moment map 
$$\mu : \OO_\vlambda \to \mathfrak k^* : (\alpha_1,\dots,\alpha_n) \mapsto \sum_i \alpha_i,$$
where again $\OO_\vlambda \defeq \OO_{\lambda_1} \times \cdots \times \OO_{\lambda_n}$. This gives the following isomorphisms:
\begin{equation} \label{eq:gkz}
V_\vlambda^{\g} \cong H^0(F_\vlambda,L_\vlambda)^{\g}
\cong H^0(\FF_\vlambda,\L_\vlambda).
\end{equation}

We will for now assume the line bundle descends and return to this condition later in the paper. In these cases we get the above geometric realisation of the invariant subspace $V^{\g}_\vlambda \subseteq V_\vlambda$ in terms of holomorphic sections of the line bundle $\L_\vlambda$ over $\FF_\vlambda$, which will makes it possible to give a geometric model in terms of differential operators acting on holomorphic sections of $\L_\vlambda$ over $\FF_\vlambda$ for the operator $\mathbf \Omega$ and thereby giving a geometrisation of the Knizhnik--Zamolodchikov connection. To this end we introduce an orthonormal basis $B$ of $\mathfrak k$.

\begin{thm} \label{prop:2ODO}
On the space of $G$-invariant sections $H^0(F_\vlambda,L_\vlambda)^{\mathfrak g}$ the action of $\Omega^{ij}$ is given by the following second order differential operator acting on sections of $\L_\vlambda$ over $\FF_\vlambda$
\begin{equation}
\Omega^{ij} = \left(\left(\sum_{\nu \in B} \nabla_{X_\nu^i} \nabla_{X_\nu^j} - \nabla_{[\pi \underline\nu^i,\pi^\bot \underline\nu^j]}\right) - i\nabla_{\underline{\alpha_i}^j + \underline{\alpha_j}^i} - \langle \alpha_i,\alpha_j\rangle\right),
\end{equation}
where $X_\nu^i$, $[\pi\underline\nu^i,\pi^\bot\underline\nu^j]$, and $\alpha_i^j$ are $K$-equivariant vector fields that are tangent to level sets of the moment map $\mu : \OO_\vlambda \to \k^*$ and thus descends to $\FF_\vlambda$ just as the functions $\langle \alpha_i,\alpha_j\rangle$ does.
\end{thm}

\begin{proof}
It suffices to show that the differential operators $\Omega^{ij}$ for $i \neq j$ are equal on $G$-invariant sections to ones defined in terms of functions and vector fields on the level set of the moment map that are $G$-equivariant. Write the infinitesimal coadjoint $\xi$-action on the $i$th component as $\underline \xi^i(\vec\alpha) = (0,\dots,0,-\ad_\xi^*\alpha_i,0,\dots,0) \in \mathfrak X(\OO_\vlambda)$ and similarly write $\underline\alpha^i = \underline \xi^i$ if $\alpha = \xi^\vee = \langle \xi,-\rangle \in \k^*$ is dual to $\xi$. Then at $\vec \alpha \in \OO_\vlambda$ the operator is given by
\begin{align}
\Omega^{ij} &= \sum_{\nu \in B} (\nabla_{\underline\nu^i} - i\alpha_i(\nu))(\nabla_{\underline\nu^j} - i\alpha_j(\nu)) \label{eq:geomkz}\\
\nonumber &= \sum_{\nu \in B} \nabla_{\underline\nu^i}\nabla_{\underline\nu^j} - i\left(\sum_{\nu \in B} \alpha_i(\nu) \nabla_{\underline\nu^j} + \sum_{\nu \in B}\alpha_j(\nu)\nabla_{\underline\nu^i}\right) - \sum_{\nu \in B}
\alpha_i(\nu)\alpha_j(\nu)\\
\nonumber &= \sum_{\nu \in B} \nabla_{\underline\nu^i}\nabla_{\underline\nu^j} - i\nabla_{\underline{\alpha_i}^j + \underline{\alpha_j}^i} - \langle \alpha_i,\alpha_j\rangle,
\end{align}
where the last equality uses that the $\nu$ are an orthonormal basis for $\g$ and $\nabla_{\underline\nu^j}$ commutes with $\alpha_i(\nu)$ because $\underline\nu^j$ is a vector field on the $j$th component while $\alpha_i$ is a function on the $i$th component and $i \neq j$ by assumption.

The functions $\valpha \mapsto \langle \alpha_i,\alpha_j\rangle$ and the vector fields $\underline{\alpha_i}^j + \underline{\alpha_j}^i$ are manifestly $K$-invariant and $(\underline{\alpha_i}^j + \underline{\alpha_j}^i)\mu(\xi) = -\alpha_j(\ad_\xi\alpha_i^\vee) - \alpha_i(\ad_\xi\alpha_j^\vee) = -\langle \alpha_j,\ad_\xi^*\alpha_i\rangle - \langle \alpha_i,\ad_\xi^*\alpha_j\rangle = 0$ by symmetry and $\ad$-invariance of $\langle \cdot,\cdot\rangle$.

That leaves only the first term. Write $\pi : TF_\vlambda \to TF_\vlambda$ for the tangential projection on the diagonal $G$-orbit and $\pi^\bot = \id - \pi$ for the orthogonal projection under the Kähler metric on $F_\vlambda$. Then $G$-invariant sections $s$ satisfy $\nabla_{\pi X}s = 0$ for any vector field $X$. As such, 
\[
\sum_{\nu \in B} \nabla_{\underline\nu^i} \nabla_{\underline\nu^j}s &= \sum_{\nu \in B} \nabla_{\pi \underline\nu^i + \pi^\bot \underline\nu^i} \nabla_{\pi \underline\nu^j + \pi^\bot \underline\nu^j} s = \sum_{\nu \in B} (\nabla_{\pi^\bot \underline\nu^i}\nabla_{\pi^\bot \underline\nu^j} + \nabla_{\pi \underline\nu^i}\nabla_{\pi^\bot \underline\nu^j})s \\
&= \sum_{\nu \in B} \left(\nabla_{\pi^\bot \underline\nu^i}\nabla_{\pi^\bot \underline\nu^j} + \nabla_{\pi^\bot \underline\nu^j}\nabla_{\pi \underline\nu^i} - \nabla_{[\pi \underline\nu^i,\pi^\bot \underline\nu^j]} - i\omega(\pi \underline\nu^i,\pi^\bot \underline\nu^j)\right)s \\
&= \sum_{\nu \in B} \left(\nabla_{\pi^\bot \underline\nu^i}\nabla_{\pi^\bot \underline\nu^j} - \nabla_{[\pi \underline\nu^i,\pi^\bot \underline\nu^j]}\right)s,
\]
where the symplectic term vanishes because the coadjoint orbits are Kähler with respect to the canonical structures \cite[8.112]{besse2007einstein} and the complex structure is preserved by the diagonal $\g$-action.

Both terms factor through the Casimir element $\sum_{\nu \in B} \nu \otimes \nu$, which means they are independent of choice of orthonormal basis $B$. It follows that the vector field $\sum_{\nu \in B} [\pi \underline\nu^i,\pi^\bot \underline\nu^j]$ is $K$-equivariant, because the adjoint action is orthogonal. In addition it preserves $\mu$ since for any $\eta \in \k$
\[
\Lie_{[\pi^\bot \underline\nu^j,\pi \underline\nu^i]}\mu(\eta) &= \omega(\underline\eta, [\pi^\bot \underline\nu^j,\pi \underline\nu^i]) = - \Lie_{\pi\underline\nu^i}\omega(\underline\eta,\pi^\bot\underline\nu^j) - \omega([\pi\underline\nu^i,\underline\eta],\pi^\bot\underline\nu^j) + (\Lie_{\pi\underline\nu^i}\omega)(\eta,\pi^\bot\underline\nu^j)\\
&= -\Lie_{\pi\underline\nu^i}0 - \omega(\pi\underline{(\ad_\eta\nu)}^i,\pi^\bot\underline\nu^j) + 0 = 0,
\]
using that Kostant--Kirillov--Souriau form is invariant and that $\omega(\underline{\g},\pi^\bot\underline\nu^j) = 0$ since $\imag\pi^\bot$ is by definition orthogonal to the diagonal $\g$-orbit with respect to the metric defined by $\omega$.

This leaves only the final term $\nabla_{\pi^\bot \underline\nu^i}\nabla_{\pi^\bot \underline\nu^j}$ to be presented in terms of $K$-equivariant vector fields that are tangent to the level sets of the moment map. The vector fields $\pi^\bot\underline\nu^i$ do not satisfy these criteria, but the operator $\nabla_{\pi^\bot \underline\nu^i}\nabla_{\pi^\bot \underline\nu^j}$ is equal to one defined in terms of vector fields that obey these conditions when acting on $G$-invariant sections. Define $e^\mu \defeq \exp(\langle \mu,-\rangle) : \OO_\vlambda \to K$ to be the exponential of Lie algebra element dual to the moment map under the Killing form. Then define the vector fields
\[
X_\nu^i \defeq \pi^\bot\underline{\Ad_{e^\mu}\nu}^i,
\]
which is manifestly $K$-equivariant using the equivariance of $\mu$. This vector field preserves $\mu$ because
\[
\Lie_{X_\nu^i} \mu(\eta) = \omega(\underline\eta,\pi^\bot\underline{(\Ad_{\exp^\mu}\nu)^i}) = 0
\]
by the same reasoning as above. Moreover,
\[
\sum_{\nu \in B}\nabla_{X_\nu^i}\nabla_{X_\nu^j} &= \sum_{\nu \in B}\nabla_{\pi^\bot\underline{\Ad_{e^{\mu}}\nu}^i}\nabla_{\pi^\bot\underline{\Ad_{e^\mu}\nu}^j} = \sum_{\nu ,\beta \in B} \nabla_{\pi^\bot\underline{(\Ad_{\exp^\mu}\nu)^i}}\nabla_{\pi^\bot\langle{\Ad_{e^\mu}\nu,\beta\rangle}\underline\beta^j} \\
&= \sum_{\nu,\beta \in B} \nabla_{\pi^\bot\underline{\Ad_{e^{\mu}}\nu}^i}\langle\Ad_{e^\mu}\nu,\beta\rangle\nabla_{\pi^\bot\underline\beta^j} \\
&= \sum_{\nu,\beta \in B} \langle\Ad_{e^\mu}\nu,\beta\rangle\nabla_{\pi^\bot\underline{\Ad_{e^{\mu}}\nu}^i}\nabla_{\pi^\bot\underline\beta^j} + \Lie_{\pi^\bot\underline{\Ad_{e^{\mu}}\nu}^i}\langle\Ad_{e^\mu}\nu,\beta\rangle\nabla_{\pi^\bot \underline\beta^j} \\
&= \sum_{\beta \in B} \nabla_{\pi^\bot\underline\beta^i} \nabla_{\pi^\bot\underline\beta^j} + 0 = \sum_{\nu \in B} \nabla_{\pi^\bot\underline\nu^i} \nabla_{\pi^\bot\underline\nu^j},
\]
because $\Lie_{\pi^\bot\underline\nu^i} \mu = 0$ and the adjoint action on $\g$ is orthonormal.

Therefore on $G$-invariant sections $s \in H^0(F_\vlambda,L_\vlambda)^\g$, the operator has the presentation
\[
\Omega^{ij}s &= \sum_{\nu \in B} (\nabla_{\underline\nu^i} - i\alpha_i(\nu))(\nabla_{\underline\nu^j} - i\alpha_j(\nu))s = \left(\left(\sum_{\nu \in B} \nabla_{X_\nu^i} \nabla_{X_\nu^j} - \nabla_{[\pi \underline\nu^i,\pi^\bot \underline\nu^j]}\right) - i\nabla_{\underline{\alpha_i}^j + \underline{\alpha_j}^i} - \langle \alpha_i,\alpha_j\rangle\right)s
\]
in terms of $G$-invariant vector fields that preserve the moment map. It follows that it descends to a second-order differential operator on $\OO_\vlambda \symp_\mu K \cong F_\vlambda \git G$.
\end{proof}

\section{The moduli space and stack of parabolic bundles}

We want to identify the space vacua $ \mathcal V_{\vlambda,k}^\dagger$, given as a sub-space of global holomorphic section of $\L_\vlambda$ over $\FF_\vlambda$ with the space of holomorphic sections of the determinant line bundle over the moduli space of semi-stable bundles. We refer to \cite{seshadri1969moduli} for the general notions of (semi)-stable parabolic bundles over marked Riemann surfaces. We only need the genus zero case here. Consider $\vec p \in C_n(\CP)$. 

\begin{definition}\index{Moduli space of parabolic vector bundles}
For a tuple $\vmu = (\mu_1,\dots,\mu_n) \in (\mathfrak t^*)^{\times n}$, denote by $\mathscr M_{\vmu}$ the moduli space of semi-stable parabolic vector bundles over $\CP$ with parabolic structure over the marked points $\vec p$ with weights and parabolic types given by $\vmu$. For dominant integral $k$-admissible weights $\vlambda \in \Lambda_k^{\times n}$ and an integer $k$, we will use the notation $\MM_{\vlambda,k} \defeq \MM_{\vmu}$ for $\vmu = \vlambda/k$.
\end{definition}

It will be useful to note that the codimension of the singular locus grows linearly with the number of parabolic points when the weights are regular. 

\begin{lemma} \label{lem:codimspace}
Let $\vlambda \in \Lambda_k^{\times n}$ be an $n$-tuple of $k$-admissible weights and assume that each $\lambda_i$ is regular. If the moduli space $\MM_{\vlambda,k}$ contains stable points, then the codimension of the locus of bundles that admit a destabilising parabolic reduction to a parabolic subgroup $P \subset G$ is at least $\frac 12(n-2)(\dim G - \dim P) - \dim Z(P)$. In particular, the codimension of the singular locus is arbitrarily large for sufficiently large $n$.
\end{lemma}

\begin{proof}
Two strictly semistable parabolic G-bundles $E_1$ and $E_2$ with destabilising parabolic reductions $\sigma_1 : \CP \to E_1/P$ and $\sigma_2 : \CP \to E_2/P$ are S-equivalent if and only if the parabolic reductions $\sigma_1^*E_1$ and $\sigma_2^*E_2$ have isomorphic projections $\Gr(E_1) \defeq p_*\sigma_1^*E_1 \cong p_*\sigma_2^*E_2 \eqdef \Gr(E_2)$, where $p : P \to L$ is the projection to the Levi subgroup (see e.g. \cite[Proposition 3.12]{ramanathan1996moduli}).

Therefore any two semistable parabolic $G$-bundles destabilised by parabolic reductions to $P \subset G$ whose parabolic reductions are isomorphic are necessarily S-equivalent. Thus the dimension of the locus of strictly semistable parabolic $G$-bundles destabilised by a parabolic reduction to a proper parabolic subgroup $P \subset G$ is upper bounded by the dimension of the moduli of parabolic $P$-bundles with inherited weights. The dimension of the moduli space of parabolic $P$-bundles with inherited weights is maximal if the inherited weights are likewise regular in $P$.

By \cite[Theorem II]{bhosle1989moduli}, for any complex reductive group $H$ the dimension of the smooth locus of the moduli space of parabolic $H$-bundles in genus 0 is equal to $\dim Z(H) + (g-1)\dim H + \sum_i \dim H/P_i$. By the assumption of regularity, $\dim G/P_{\lambda_i} = \frac 12(\dim G - r)$ and $\dim P/P_{\mu_i} \leq \frac 12 (\dim P - r)$ for $r = \rk G = \rk P$, since $P \subset G$ contains the Borel subgroup, which contains a maximal torus.

Then the codimension is bounded from below by
\[
\dim \mathcal{M}_G - \dim \mathcal{M}_P &\geq \dim Z(G) - \dim G + \frac 12\sum_i (\dim G - r) - \dim Z(P) + \dim P - \frac 12\sum_i (\dim P - r) \\
&= \frac 12(n-2)(\dim G - \dim P) - \dim Z(P),
\]
which is uniformly unbounded in $n$ for all proper parabolic subgroups.
\end{proof}

We need the stack point of view on holomorphic $G$ bundles in order to establish the required results on Picard groups and to have the tools to setup the basis for our the main theorem. We follow the setup and notation of \cite{laszlo1997line} and \cite{teleman2003parabolic} closely. Let $P_i$ be parabolic subgroups $i=1, \ldots, n$ and let $\vec P = (P_1, \ldots, P_n)$.

Denote by $\M_{\vec P} = \M_{\vec P}(\CP,\vec p)$ \emph{the moduli stack of quasi-parabolic bundles of type $\vec P$ over $(\CP,\vec p)$}. For weights $\vmu$ denote likewise the moduli stack of parabolic bundles as $\M_{\vmu}$. Moreover, the \emph{moduli stack of $\vmu$-semi-stable bundles} $\M_{\vmu}^{\rm ss} \subset \M_{\vmu}$, is the substack consisting of quasi-parabolic bundles that are semi-stable parabolic bundles when equipped with weights $\vmu$. We correspondingly use the notation $\M_{\vlambda,k} \defeq \M_{\vmu}$ for $\vmu = \vlambda/k$.

\begin{thm}[{\cite[Theorem 2.7]{teleman2003parabolic}}]
The moduli space of parabolic bundles is a moduli space for the stack of semi-stable parabolic bundles and the open subspace of stable parabolic bundles is a moduli space for the stack of stable parabolic bundles.
\end{thm}

Let $LG$ be the loop group of $G$ and let $L^+G$ be the subgroup of loops extending over the formal disk and consider 
\[
\mathfrak Q \defeq LG/L^+G.
\]
as defined in \cite[3.6]{laszlo1997line}. The identity in $LG$ gives a base point $1\in \mathfrak Q$.

Set $X_{\vec p} = \CP - \vec p$ and let $G[X_{\vec p}]$ denote the group of algebraic maps from $X_{\vec p}$ to $G$.

\begin{thm}[{
\cite[Theorem 8.5]{laszlo1997line}}] \label{thm:uniformisation}
The stack of (quasi-)parabolic $G$-bundles of type $\vec P$ over $(\CP,\vec p)$ can be realised as
\[
\M_{\vec P} \cong G[X_{\vec p}] \backslash (\mathfrak Q \times F_\vlambda).
\]
Moreover, the quotient map induces an isomorphism of Picard groups giving an identification
$$\Pic(\M_{\vec P} ) \cong \Z \times \prod_{i=1}^n\mathcal X(P_{i}),$$ 
where $\mathcal X(P_{i})$ denotes the character group of $P_{i}$.
\end{thm}
For $(\ell,\vmu) \in \Z \times \prod_{i=1}^n \mathcal X(P_{\lambda_i})$ the corresponding line bundle is denoted by $\mathfrak L_{(\vmu,\ell)}$. 
In order to applying the metaplectic-corrected construction of the Hitchin connection we need to know when a square root of the canonical bundle exist and identify it.

\begin{corollary} \label{cor:metaplec}
The canonical bundle over the moduli stack $\M_B$ is $\mathfrak K = \mathfrak L_{(-2\vec\rho,-2h)}$, thus it admits a square root
\[
\delta = \mathfrak L_{-\vec\rho,-h},
\]
where $\vec\rho = (\rho,\dots,\rho) \in \mathcal X(B)^{\times n}$ and $\rho = \frac 12 \sum_{\alpha \in \Phi^+} \alpha$ is the Weyl vector.
\end{corollary}

\begin{proof}
By \Cref{thm:uniformisation} this follows from the fact that the canonical line bundle of $\mathfrak Q$ corresponds to the element $(-2h,0,\dots,0)\in \Z \times \prod_{i=1}^n \mathcal X(B) \cong \Pic(\M_B)$ \cite[Theorem 5.4]{kumar1997infinite} and the fact that the canonical bundle of $F = {G}/B$ corresponds to $-2\rho$ in $\X(B)$ since 
\[
K_F \cong \bigotimes_{\alpha \in \Phi^+} L_{-\alpha} = L_{-2\rho}.
\]
As $\rho \in \mathcal X(B)$, we see that $(-h,-\vrho) \in \Pic(\M_B)$ defines the unique square root $\delta$ of the canonical bundle.
\end{proof}

We consider the inclusion $\iota : \M^{\rm ss}_{\vlambda,k} \into \M_{P_\vlambda}$.

\begin{prop}
Let $\vlambda,\vmu$ be $n$-tuples of integral dominant weights. If each $\lambda_i$ is regular and $n$ is large enough then $\MM_{\vlambda,k}$ has stable bundles. If moreover the character $\sum_i\mu_i$ defines the trivial representation $Z(G) \to \C^\times$, then the line bundle $\iota^*\mathfrak L_{\vmu,\ell}$ descends to a line bundle $\mathscr L_{\vmu, \ell}$ over $ \mathscr M_{\vlambda,k}$.
\end{prop}

\begin{proof}
By \cite[Théorème 2.3]{drezet1989groupe}, a linearised line bundle descends if and only if the stabiliser of each point acts trivially on the fibre over that point. By \Cref{lem:codimspace} and Hartogs' Theorem it is sufficient for large enough $n$ that the line bundle descends over the stable bundles, whose stabiliser is the centre $Z(G)$. The character $\sum_i\mu_i$ gives precisely the linearised action of $Z(G)$ on $\mathfrak L_{\vmu,\ell}$, which by hypothesis is trivial.
\end{proof}

In the specific case of $G = \SL(r)$, a stronger result can be found in \cite[Théorème~3.3]{pauly1996espaces}, without the assumption on the size of $n$ or the regularity of $\vlambda$.

\begin{lemma} \label{lem:evenchern}
If $n$ is even and the vector of weights $\vmu \in \Lambda_k^{\times n}$ are all regular, then the moduli space $\MM_{\vmu}$ admits a square root of its canonical bundle.
\end{lemma}

\begin{proof}
The flag type of regular weights is the maximal flag type and by \Cref{cor:metaplec}, the moduli stack of parabolic bundles for the maximal flag type admits the unique square root $\mathfrak L_{-h,-\vrho} = \mathfrak K^{1/2}$ of its canonical bundle. The canonical bundle of the moduli space pulls back to the canonical bundle of the moduli stack, so it follows that $2\sum_{i=1}^n \rho = 2n\rho : Z(G) \to \C^\times$ is trivial, irrespective of $n$. As such, $2\rho$ defines a trivial representation $Z(G) \to \C^\times$ and therefore $\mathfrak L_{-h,-\vrho}$ descends to the moduli space whenever $n$ is even.
\end{proof}

\section{The needed isomorphisms and injections}\label{iso}

In this section we establish all the required isomorphism and injections between (sub-)spaces of global holomorphic sections of various holomorphic line bundles over various symplectic quotients, moduli spaces and stacks.

The following theorem related the spaces of holomorphic sections over moduli space of parabolic bundles and over the stack of parabolic bundles.

\begin{thm}[{\cite[Theorem 9.6]{teleman2000quantization}}] \label{thm:stackisspace}
For any $\vlambda \in \Lambda_k^{\times n}$ for which $\mathfrak L_{\vlambda,k}$ desends to $\mathscr M_{\vlambda,k}$ we have that
\[
H^0(\M_{\vlambda,k},\mathfrak L_{\vlambda,k}) \cong H^0(\mathscr M_{\vlambda,k}, \mathscr L_{\vlambda,k})
\]
and for $i > 0$
\[
H^i(\M_{\vlambda,k},\mathfrak L_{\vlambda,k}) \cong H^i(\mathscr M_{\vlambda,k}, \mathscr L_{\vlambda,k}) = 0.
\]
\end{thm}

\begin{thm}\label{thm:equiv}
Let $\vlambda \in \Lambda_k^{\times n}$. If the space of conformal blocks $\mathcal V_{\vlambda,k}^\dagger $ is non-zero, then
\begin{enumerate}
\item The Bott--Borel--Weil bundle $L_\vlambda \vb F_\vlambda$ descents to the GIT quotient $\L_\vlambda \vb \FF_\vlambda$, with this line bundle given by $\L_\vlambda \defeq (L_\vlambda \git G)$ over $\FF_\vlambda \defeq (F_\vlambda \git G)$;
\item The centre of the group acts trivially on $V_\vlambda$ and $\sum_i \lambda_i$ is in the root lattice;
\item The line bundle $\mathfrak L_{\vlambda,k} \vb \M_{\vlambda,k}$ descends from the stack to the moduli space $\mathscr L_{\vlambda,k} \vb \mathscr M_{\vlambda,k}$. 
\end{enumerate}
\end{thm}

\begin{proof}
Since we assume $\mathcal V_{\vlambda,k}^\dagger \neq 0$ it follows in particular that $0 \neq V_\vlambda^{\g} \cong H^0(F_\vlambda,L_\vlambda)^{\g}$. As such, $L_\vlambda \vb F_\vlambda$ has a non-zero $\g$-invariant section, which therefore descends to the GIT quotient and defines a line bundle, meaning that $L_\vlambda$ descends.

By Kempf's Descent Lemma \cite[Théorème 2.3]{drezet1989groupe}, this implies that stabilisers act trivially on the line bundle $L_\vlambda \vb F_\vlambda$. Since the centre acts trivially on the flag variety, it follows that it acts trivially on $L_\vlambda$ and therefore on $V_\vlambda$. The centre acts via the character $\sum_i\lambda_i$, which must therefore be in the root lattice, since the quotient of the weight lattice by the root lattice is the center of $G$ \cite[p.~373]{fulton2013representation}.

By Pauly's isomorphism $H^0(\M_{\vlambda,k},\mathfrak L_{\vlambda,k}) \neq 0$ so there is at least one isomorphism-invariant section of $\mathfrak L_{\vlambda,k}$, which therefore descends to the moduli space to generate the line bundle $\mathscr L_{\vlambda,k} \to \MM_{\vlambda,k}$.

\end{proof}

Given a tuple of weights $\vmu = \vlambda/ k$, each flag structure defines a (not necessarily semi-stable) parabolic bundle structure on the trivial $G$-bundle. Since semi-stability is an open condition in both the bundle and the GIT sense, and the set of parabolic bundles with trivial underlying bundle is Zariski-open in the moduli space, there is a Zariski-open subset $\tilde U_{\vlambda,k} \subseteq F_\vlambda$ of flags that are semi-stable in both senses. Write $F_\vlambda^{\rm ss} \subseteq F_\vlambda$ for the subset of flags that are semi-stable in the GIT sense.

Since, by \cite[Proposition 4.2(b)]{teleman2003parabolic} semi-stable bundles of type $(\vlambda,k)$ exist if and only if semi-stable bundles based on the trivial bundle exist, and further flag configurations that give semistable parabolic bundle structures on the trivial bundle are in particular GIT-semistable, as GIT-semistability coincides with the absence of destabilising trivial bundles for parabolic structures on the trivial bundle by \cite[Remark 4.3]{teleman2003parabolic}, we see that the subset $\tilde U_{\vlambda,k}$ is non-empty whenever $\mathscr M_{\vlambda,k}$ is non-empty. The quotient $U_{\vlambda,k}$ of $\tilde U_{\vlambda,k}$ by S-equivalence therefore defines a non-empty Zariski-open subset of both the moduli space $\mathscr M_{\vlambda,k}$ and the GIT quotient $\FF_\vlambda = F_\vlambda \git G$. This is summarised in the commutative diagram below.
\begin{equation} \label{diag:stacky}
\begin{tikzcd}
\mathfrak M_{P_{\vec \lambda}}                                          &  & {G[X_{\vec p}] \backslash (\mathfrak Q_G \times \mathcal F_\vlambda)} \arrow[ll, "\cong"'] &  & \mathfrak Q_G \times  F_\vlambda \arrow[ll, ""', two heads] \\
{\mathfrak M_{\vec \lambda,k}^{\rm ss}} \arrow[u, hook] \arrow[d, two heads] &  & {\tilde U_{\vlambda,k}} \arrow[rr, hook] \arrow[d, two heads] \arrow[ll]                           &  & F^{\rm ss}_\vlambda \arrow[u, "{(1,-)}"', hook'] \arrow[d, two heads]                  \\
{\mathscr M_{\vec \lambda,k}}                                            &  & {U_{\vec \lambda,k}} \arrow[ll, "q"'] \arrow[rr, "\iota", hook]                        &  & \FF_\vlambda                                                                              
\end{tikzcd}
\end{equation}

Since the map $\tilde U_{\vlambda,k} \to \mathscr M_{\vlambda,k}$ is manifestly $G$-invariant, as the group action does not change the isomorphism type, it induces a map $q : U_{\vlambda,k} \into \mathscr M_{\vlambda,k}$ with $U_{\vlambda,k} = \tilde U_{\vlambda,k} \git G$ by the universal property, which restricts to an isomorphism with the stable parabolic bundles based on the trivial bundle.

\begin{lemma} \label{lem:lineiso}
If the space of conformal blocks does not vanish, Diagram~\ref{diag:stacky} induces an explicit isomorphism $q^* \mathscr L_{\vlambda,k} \cong \iota^*\L_\vlambda$.
\end{lemma}

\begin{proof}
The inclusion $(1,-) : F_\vlambda \into \mathfrak Q_G \times F_\vlambda$ induces via pullback a projection on the second coordinate of Picard groups $\Pic(\mathfrak Q_G) \times \Pic(F_\vlambda) \vb \Pic(F_\vlambda)$. By \Cref{thm:uniformisation}, the quotient map to the stack induces an isomorphism of Picard groups and therefore $\mathfrak L_{(\vlambda,k)} \vb \M$ pulls back to the restriction of the Bott--Borel--Weil bundle $L_\vlambda \vb F_\vlambda^{\rm ss}$, corresponding to the element $\vlambda \in \Pic(F_\vlambda) \cong \prod_{i=1}^n \mathcal X(P_{\lambda_i})$. Therefore, by commutativity of the diagram, the pullbacks of the restrictions of $\mathfrak L_{\vlambda,k}$ and $L_\vlambda$ pull back to the same line bundle over $\tilde U_{\vlambda,k}$.

Under the assumption on the conformal blocks, these line bundles descend to $\MM_{\vlambda,k}$ and $\FF_{\vlambda}$ by \Cref{thm:equiv}. Thus $\iota^*(\L_\vlambda) \cong q^* \L_{\vlambda,k}$ as line bundles over $U_{\vlambda,k}$, by commutativity of the diagram.
\end{proof}

\begin{lemma} \label{lem:quantiso}
If $\mathcal V^\dagger_{\vlambda,k} \neq 0$, then the identification $q^* \mathscr L_{\vlambda,k} = \iota^*\L_\vlambda \eqdef \mathtt L_\vlambda$ from \Cref{lem:lineiso} gives the equality 
$$\imag q^* = \iota^*[\mathcal V_{\vlambda,k}^\dagger],$$
under the injections
\[
q^* &: H^0(\MM_{\vlambda,k},\L_{\vlambda,k}) \to H^0(U_{\vlambda,k},\mathtt L_\vlambda), \ \ \ \ \  \iota^* : H^0(\FF_{\vlambda},\mathcal L_{\vlambda}) \to H^0(U_{\vlambda,k},\mathtt L_\vlambda),
\]
where $\mathcal V_{\vlambda,k}^\dagger \subset H^0(\FF_\vlambda,\L_\vlambda)$.
\end{lemma}

\begin{proof}
By \cite[Theorem 9.6]{teleman2000quantization}, the maps $\mathscr M_{\vlambda,k} \twoheadleftarrow \M^{\rm ss}_{\vlambda,k} \hookrightarrow \M_{P_\vlambda}$ induce isomorphisms
\[
H^0(\mathscr M_{\vlambda,k},\mathscr L_{\vlambda,k}) \cong H^0(\M^{\rm ss}_{\vlambda,k},\mathfrak L_{\vlambda,k}^{\rm ss}) \cong H^0(\M_{P_\vlambda},\mathfrak L_{\vlambda,k}).
\]
By \cite[Theorem 1.2.1]{laszlo1997line}, the pullback along $\mathfrak Q_G \times F_\vlambda \to \M_{P_\vlambda}$ is an isomorphism from the space of sections $H^0(\M_{P_\vlambda},\mathfrak L_{\vlambda,k})$ to the conformal blocks $\VV_{\vlambda,k}^\dagger$. The pullback over $(1,-)$ is then immediately $G$-invariant (again because the $G$-action does not change the isomorphism type and thus is trivial on the moduli space) and must descend to the GIT quotient. The commutativity of the Diagram~(\ref{diag:stacky}) then implies that the sections pulled back to $U_{\vlambda,k}$ must be identical, proving the result.
\end{proof}

We have thus established the following theorem.

\begin{thm} \label{thm:pauly}
The maps $q$ and $\iota$ instantiate Pauly's isomorphism for the Bott--Borel--Weil model of the space of conformal blocks:
$$ \Phi := (q^*)\inv \circ \iota\mid_{\mathcal V_{\vlambda,k}^\dagger }^* : \mathcal V_{\vlambda,k}^\dagger \rightarrow H^0(\mathscr M_{\vlambda,k},\mathscr L_{\vlambda,k}). $$
\end{thm}

This isomorphism of vector spaces is natural in the complex structure and therefore defines an isomorphism of vector bundles over $ {\mathbb M}_n$ between the sheaf of conformal blocks and the Verlinde bundle $H_{\vlambda,k} \vb {\mathbb M}_n$ whose fibres are given by $H^0(\mathscr M_{\vlambda,k},\mathscr L_{\vlambda,k})$.

We will need to shift weights and levels in order to construct the Hitchin connection in \Cref{HC}.

\begin{thm} \label{thm:meta-incl}
For $\vlambda,\vec\rho=(\rho,\dots,\rho) \in \Lambda_k^{\times n}$ there is a natural inclusion of vector spaces $$\tilde \pi: H^0(\mathscr M_{\vlambda,k},\mathscr L_{\vlambda, k}) \into H^0(\mathscr M_{\vlambda+ \vec\rho,k+h},\mathscr L_{\vlambda+\vrho,k+h}\otimes \delta).$$ 
\end{thm}

\begin{proof}
The projections $G/B \vb G/P_i$ induce a projection $\pi : \M_B \vb \M_{P_\vlambda}$ via the presentation from \Cref{thm:uniformisation}, giving rise to the following diagram.
\begin{center}
\begin{tikzcd}
\mathfrak M_{P_{\vec \lambda}}                                          &                                                                             & \mathfrak M_B \arrow[ll, "\pi"', two heads]                                                   \\
{\mathfrak M_{\vec \lambda,k}^{\rm ss}} \arrow[u, hook] \arrow[d, two heads] &                                                                             & {\mathfrak M_{\vec \lambda + \vec \rho,k + h}^{\rm ss}} \arrow[u, hook] \arrow[d, two heads] \\
{\MM_{\vec \lambda,k}}                                            &
& {\MM_{\vec \lambda + \vec \rho,k + h}}                                           
\end{tikzcd}
\end{center}
The pullback bundle $\pi^*\mathfrak L_{\vlambda,k} \vb \M_B$ of $\mathfrak L_{\vec \lambda,k} \vb \M_{P_\vlambda}$ is isomorphic to $\mathfrak L_{\vlambda,k} \vb \M_B$, as the projection $G/B \vb G/P_{\lambda_i}$ induces an inclusion $\mathcal X(P_{\lambda_i}) \into \mathcal X(B)$ because each character of $P_{\lambda_i}$ defines a (unique) character on $B \subseteq P$, so $\pi^* : \Pic(\M_{P_\vlambda}) \into \Pic(\M_B)$ is an inclusion via \Cref{thm:uniformisation}.

By \Cref{cor:metaplec}, the unique square root of the canonical bundle of $\M_B$ is $\mathfrak L_{-\vrho,-h}$. Since the projection $\M^{\rm ss}_{\vlambda + \vec \rho,k + h} \to \mathscr M_{\vlambda + \rho,k + h}$ maps the canonical bundle to the canonical bundle, it follows that the pull back of $\mathscr L_{\vec \lambda + \vec \rho, k + h} \otimes \delta$ coincides with the restriction of $\mathfrak L_{\vlambda + \vrho, k + h} \otimes \mathfrak L_{-\vrho,-h} \cong \mathfrak L_{\vlambda,k}$.

Any holomorphic section of $\mathscr L_{\vlambda,k} \vb \mathscr M_{\vlambda,k}$ pulls back to one of $\mathfrak L_{\vlambda,k}|_{\M^{\rm ss}_{\vlambda,k}} \vb \M_{\vlambda,k}^{\rm ss}$ by definition of $\mathscr L_{\vlambda,k}$, which extends uniquely to a section on $\mathfrak L_{\vlambda,k} \vb \M_{P_\vlambda}$ by \Cref{thm:stackisspace}. The pullback along $\pi$ can then be restricted to the semi-stable substack $\M_{\vlambda + \rho,k + h}^{\rm ss} \subset \M_B$, which descends uniquely to the moduli space $\mathscr M_{\vlambda + \vrho,k+h}$ again via \Cref{thm:stackisspace}. This defines a natural map, which is injective because the pullback $\pi^*$ is injective and all other maps are isomorphisms of spaces of sections.
\end{proof}

Again, this map is natural in the complex structure and hence extends to an inclusion of vector bundles
$$ H_{\vlambda,k}  \subset  \updel H_{\vlambda+\vrho,k+h}$$
where $ \updel H_{\vlambda+\vrho,k+h}\vb {\mathbb M}_n$ is the vector bundle over $\mathbb{M}_n$ whose fibers are given by $H^0(\mathscr M_{\vlambda+\vrho,k+h},\mathscr L_{\vlambda+\vrho,k+h}\otimes \delta)$.

Likewise, combining \Cref{thm:meta-incl} with \Cref{thm:pauly} gives a similar map for the sheaf of conformal blocks.

\begin{thm}\label{incl1}
There is a natural inclusion of vector bundles over $\mathbb{M}_n$
$$ \Psi := \tilde \pi \circ \Phi : \mathcal V_{\vlambda,k}^\dagger \into  \updel H_{\vlambda+\vrho,k+h}. $$
\end{thm}

Theorem \ref{thm:meta-incl} and \ref{incl1} will be essential for us when we construct the Hitchin connection and compare it to the TUY-connection in $V_{\vlambda,k}^\dagger $.

For the construction of the metaplectic correction, it will be important to understand the growth of the codimension of the singular locus with the number of marked points.

\begin{lemma} \label{cor:metaexists}
Assume $n \geq 3$ and $\vlambda \in \Lambda_k^{\times n}$ a tuple of $k$-admissible weights. For a positive integer $m$ let $\vec \lambda_m = (\vlambda,\vec0) \in \Lambda_k^{\times (n+m)}$ be $\vlambda$ followed by $m$ zeros and $\vec\rho=(\rho,\dots,\rho) \in \Lambda_k^{\times (n+m)}$. Then for sufficiently large $m$ 
$$ H^{0,1}(\mathscr M^{\rm s}_{\vlambda_m+\vrho,k + h}) = 0.$$
\end{lemma}

\begin{proof}
For sufficiently large $m$ there are always stable bundles, so by \Cref{lem:codimspace} the codimension of the complement scales positively with $(n+m-2)$. So for sufficiently large $m$, Hartogs' Theorem implies that $H^{0,1}(\mathscr M^{\rm s}_{\vlambda_m+\vrho,k + h}) \cong H^{0,1}(\mathscr M_{\vlambda_m+\vrho,k + h})$. Since $\mathscr M_{\vlambda_m+\vrho,k + h}$ is rational and both $\mathscr M_{\vlambda_m+\vrho,k + h}$ and projective space are proper, it follows they are properly birational, witnessed by the closure of the graph of the birational map.

Therefore \cite[Corollary 5]{lodh2024birational} implies $H^{0,1}(\mathscr M_{\vlambda_m+\vrho,k + h}) \cong H^{0,1}(\CP[r]) = 0$, since $\mathscr M_{\vlambda_m+\vrho,k + h}$ is of finite type and therefore Nagata, and because the singular locus has arbitrarily high codimension the conditions S$_3$ from \cite[Corollary 5]{lodh2024birational}  and pseudo-rationality in codimension 2 are vacuously satisfied.
\end{proof}

\section{The moduli space of flat connections}

Since we want to identify the Hitchin connection construction in the bundle over Teichmüller space, whose fibres are obtained by applying geometric quantisation to the moduli spaces of flat connections, we now consider the moduli space of flat connections over a smooth surface $S^2$ of genus zero with marked points with respect to the compact simple real Lie group $K$. Principal bundles over surfaces with simply connected structure groups are trivialisable, so we may and will assume we are considering the trivial $K$-bundle over $S^2$ and identify connections with connection 1-forms. Fix a finite set of marked points $\vec s = (s_1, \ldots s_n)$ on $S^2$ and $\valpha = (\alpha_1,\dots,\alpha_n) \in \mathfrak t^{\times n}$. Let $\Sigmao$ be the complement of the points $\vec s$.

\begin{definition}
The \emph{moduli space $\mathcal{M}_{\vec \alpha}$ of flat $K$-connections with holonomy $\vec\alpha \in \k^{\times n}$} is defined to be 
\[
\mathcal{M}_\valpha \defeq \{A \in \Omega^1(\Sigmao,\k) \con F_A = 0 \andd [\Hol_{A,i}] = [\exp(2\pi i\alpha_i)]\} / \mathcal K,
\]
where $\Hol_{A,i}$ is the the holonomy of $A$ around $s_i$, $[\cdot]$ denoting conjugacy class, $F_A$ the curvature and $\mathcal K = C^\infty(\Sigmao,K)$ is the gauge group.

The \emph{irreducible locus} $\mathcal{M}'_{\vec \alpha} \subseteq \mathcal M_{\vec \alpha}$ is the subset of equivalence classes of connections whose stabiliser in $\mathcal K$ is the centre of $K$. For an integer $k \geq 1$ and a tuple of dominant integral weights $\vlambda \in \Lambda_k^{\times n}$, write $\mathcal M_{\vlambda,k} \defeq \mathcal M_{\vlambda/k}$, where the canonical isomorphism $\mathfrak t \cong \mathfrak t^*$ via the normalised Killing form has been used to identify weights and coweights.
\end{definition}

The irreducible locus is a smooth manifold, while the moduli space itself may have singularities at the complement, the reducible locus. 

\begin{lemma}[{\cite[Lemma~2.2]{biswas1993principal}}]\label{lem:tangent}
At an irreducible connection $A$, the tangent space $T_{[A]}\mathcal M'_{\vlambda,k}$ is isomorphic to the image of $H^1_c(\Sigmao,\d_A)$ in $H^1(\Sigmao,\d_A)$, where $H_c^i$ denotes the compactly supported de Rham cohomology.
\end{lemma}

This description of the tangent bundle to $\mathcal M'_{\vlambda,k}$ and the fixed inner product in the Lie algebra induces the Atiyah--Bott symplectic form $\omega_{\vlambda,k}$ on $\mathcal M'_{\vlambda,k}$. See also \cite{biswas1993principal} in the genus zero case, where the representation variety version of this moduli space
\[
\mathcal M_{\vlambda,k} \cong \big\{\rho \in \Hom(\pi_1(\Sigmao), K) \con \forall i : \rho(S_i) \in [ \exp(2\pi i\lambda_i/k)]\big\} / K
\]
is used to present a finite dimensional group cohomology construction of $\omega_{\vlambda,k}$.

\begin{thm}[{\cite[Section 3.3]{MW}, \cite[Theorem 1.1]{Charles2013}}]
Let $\vlambda \in \Lambda_k^{\times n}$  such that $\mathcal M_{\vlambda,k}' \neq \emptyset$. If $\sum_i\lambda_i$ is in the root lattice, then the Chern-Simon bundle construction gives a pre-quantum line bundle $({\mathcal L}_{\vlambda,k}, \nabla, \langle\cdot,\cdot\rangle)$ over $\mathcal M'_{\vlambda,k}$, e.g. the curvature of the connection $\nabla$, which is Hermitian with respect to $\langle\cdot,\cdot\rangle$ is
$$ F_\nabla = -i \omega_{\vlambda,k}.$$ 
Furthermore, the action of the mapping class group $\Gamma$ lifts to ${\mathcal L}_{\vlambda,k}$.
\end{thm}

\section{The Teichmüller family of complex structures}

We will consider the natural family of complex structures on $\mathcal M'_{\vlambda,k}$ parametrised by the appropriate version of Teichmüller space, namely the one used in \cite{AU4}. This requires fixing directions $\vec v = (v_1, \ldots v_n)$, where $v_i \in (T_{s_i}S^2 -\{0\})/{\mathbb R}_+$ and denote by $\T = \T_{(S^2,\vec s,\vec v)}$ the Teichmüller space of $(S^2, \vec s, \vec v)$. See Definition~2.1 in \cite{AU4}. We have a natural projection map $p : \T \rightarrow {\mathbb M}_n$. The mapping class group $\Gamma$ of $(S^2, \vec s, \vec v)$ acts on $\T$, $p$ is invariant under this action and indeed induces an isomorphisms $\T/\Gamma \cong {\mathbb M}_n$. When we consider the vector bundle  $p^*\mathcal V_{\vlambda,k}^\dagger$ over $\T$, we will omit $ p^*$ and simply write $\mathcal V_{\vlambda,k}^\dagger$. For $n_1 < n_2$ we have the forgetful map from $M_{n_2}$ to $M_{n_1}$ which is covered by the corresponding maps of Teichmüller spaces further compatible with natural group homomorphisms of the respective mapping class groups.

Let us now recall the identification of the moduli space of flat $K$-connections with the moduli space of parabolic holomorphic $G$-bundles. So far our notation for the moduli space of semi-stable parabolic holomorphic $G$-bundles has suppressed the dependence of the choice of a point in ${\mathbb M}_n$, but for any $\sigma\in \T$ we will use the notation $\mathscr M_{\vlambda,k, \sigma}$ when we are considering this moduli space with respect to $p(\sigma)$.

\begin{thm}[{Mehta--Seshadri Theorem \cite[Theorem 4.1]{mehta1980moduli}, \cite[Théorème 2.5]{biquard1991fibres}}] \label{thm:MS}
For each $\sigma \in \T$ there is a homeomorphism
\begin{equation} \label{eq:MS}
\Phi^{}_{\vlambda,k}( \sigma) : \mathcal M_{\vlambda,k} \to \mathscr M_{\vlambda,k, \sigma}^{}
\end{equation}
which restricts to a diffeomorphism on the smooth locus, given by assigning to each flat $K$-connection the corresponding semi-stable holomorphic bundle structure on the trivial $G$-bundle with respect to $\sigma$ together with the parabolic subgroups $P_\vlambda$ at $\vec p= p(\vec s)$. 
\end{thm}

We can pull back the complex structure on $\mathscr M_{\vlambda,k}$ via $\Phi^{}_{\vlambda,k}( \sigma)$ to get a complex manifold structure on $\mathcal M'_{\vlambda,k}$, which we denote $\mathcal M'_{\vlambda,k,\sigma}$. This creates a family of Kähler structures on $(\mathcal M'_{\vlambda,k,\sigma},\omega_{\vlambda,k})$ parametrised by $\sigma \in \T$.
By considering the $(0,1)$-part of the prequantum connection with respect to each $\sigma$, we get a holomorphic line bundle structure on ${\mathcal L}_{\vlambda,k}$, which we correspondingly denote ${\mathcal L}_{\vlambda,k,\sigma}$ over $\mathcal M'_{\vlambda,k,\sigma}$, also parametrised by $\sigma \in \T$.

\begin{lemma}[{\cite[Theorem 5.8]{daskalopoulos2011geometric} \cite[Corollary 4.3]{andersen2017witten}}] \label{lem:CS=det}
We have a natural isomorphism of holomorphic line bundles
$$ \Phi^{}_{\vlambda,k}( \sigma) ^* ({\mathscr L}_{\vlambda,k,\sigma})\mid_{\mathcal M'_{\vlambda,k,\sigma}} \cong {\mathcal L}_{\vlambda,k,\sigma}.$$
\end{lemma}

We will also denote by  $H_{\vlambda,k}$ the pull back along $p$ of the bundle $H_{\vlambda,k}$, now regarding as a bundle over $\T$ whose fiber at $\sigma$ is the vector space
$$ H_{\vlambda,k, \sigma} = H^0(\mathscr M_{\vlambda,k, \sigma}, {\mathscr L}_{\vlambda,k, \sigma}) \subset H^0(\mathcal M'_{\vlambda,k, \sigma}, {\mathcal L}_{\vlambda,k, \sigma}) $$
where the inclusion is an equality if the codimension of the singular locus in $\mathcal M_{\vlambda,k}$ is strictly greater than $1$. We observe that there is an action of the mapping class group on bundle $H_{\vlambda,k}$ covering the action of the mapping class group on $\T$. 

As we saw in \Cref{iso}, for $m$ large enough the singular locus of ${\mathscr M}_{\vlambda_m+\vec \rho,k+h}$ is of codimension at least $2$, so we see that
$$ H^0(\mathcal M'_{\vlambda_m+\vec \rho,k+h,\sigma}, {\mathcal L}_{\vlambda_m+\vec \rho,k,\sigma}) \cong H^0({\mathscr M}_{\vlambda+\vec \rho,k+h, \sigma}, {\mathscr L}_{\vlambda+\vec \rho,k+h,\sigma})$$
for each $\sigma \in \T$.
Observe that ${\mathscr M}_{\vlambda_m+\vec \rho,k+h, \sigma}$ has a unique square root $\delta_{\sigma}$ (its structure as a holomorphic line bundle depends on $\sigma$, thus the subscript), which via $ \Phi^{}_{\vlambda,k}( \sigma)$ pulls back to a square root of the canonical bundle of $\MM'_{\vlambda+\vec \rho,k+h,\sigma}$, which we also denote $\delta_{\sigma}$. Similarly, we also denote by $\updel H_{\vlambda+\vrho,k+h}$ the pull back of the  holomorphic vector bundle $\updel H_{\vlambda+\vrho,k+h}$ along $p$ and the fibre at $\sigma$ is of course given by
\begin{equation} \label{eq:metaplecticbundle}
 \updel H_{\vlambda+\vrho,k+h, \sigma} :=  H^0(\MM'_{\vlambda+\vec \rho,k+h,\sigma}, {\mathscr L}_{\vlambda +\vrho,k+h,\sigma}\otimes \delta_\sigma).
\end{equation}

Combining the isomorphism and inclusions from \Cref{iso} we arrive at the following theorem.

\begin{thm}\label{incl2}
There are natural inclusion of bundles over $\T$ 
$$ \tilde \pi: H_{\vlambda,k} \into \updel H_{\vlambda+ \vec\rho,k+h} \ \ {\rm and } \ \ \Psi : {\mathcal V}_{\vlambda,k}^\dagger \rightarrow \updel H_{\vlambda+ \vec\rho,k+h} $$
which are compatible with the action of the mapping class groups.
\end{thm}

Note that we in the formulation of the above theorem have suppressed the need pull-backs between the relevant Teichmüller spaces and the corresponding homomorphism of the relevant mapping class groups. Further by Theorem \ref{thm:pauly} the images of the these two inclusions agree.

\section{The metaplectic-corrected Hitchin connection} \label{HC}

In \cite{andersen2012hitchin} a generalisation of the of the original Hitchin contruction in \cite{hitchin1990flat} was given, which slightly reduced the differential geometric assumptions for the existence of a Hitchin connection. However, the proportionality of the first Chern class and the class of the symplectic form prevents the application of the main theorem of \cite{andersen2012hitchin} to be applied to moduli spaces of flat connections on punctures surfaces. 

To overcome this \cite{andersenmetaplectic} proposed a construction via the metaplectic correction to circumvent this topological obstruction. We briefly review this construction for a general symplectic manifold $(M,\omega)$ with a prequantum line bundle $(L, \nabla, \langle \cdot,\cdot\rangle)$, where $\langle \cdot,\cdot\rangle$ is a Hermitian structure in $L$ preserved by the connection $\nabla$ whose curvature is
$$ F_\nabla = - i\omega.$$
Fix a family of complex structures on $M$, which are all compatible with $\omega$, parametrised by a manifold $\T$. For $\sigma \in \T$ denote by $M_\sigma$ the resulting Kähler manifold. Under the assumption that the second Stiefel--Whitney class $w_2(M)$ vanishes and $H^{0,1}(M_\sigma) = 0$ for all $\sigma \in \mathcal T$, there exists a unique square root $\delta_\sigma$ of the canonical bundle $K_\sigma$ of $M_\sigma$ for all $\sigma \in \T$, which has an induced connection from the Levi-Civita connection, whose curvature is $i\rho_\sigma$, where the (1,1)-form $\rho_\sigma$ is the Ricci form. Consider then the subspaces
$$ H^0(M_\sigma, L^{k}_\sigma\otimes \delta_\sigma) \subset C^\infty(M, L^k \otimes \delta_\sigma)$$
for each $\sigma\in \T$. The moduli spaces applications which we have in mind all satisfies that these subspaces are all finite dimensional, thus we will now assume this and denote by $\updel H^{(k)} \vb \T$ the resulting subbundle of the trivial bundle $\updel \mathcal H^{(k)} \defeq C^\infty(M, L^k \otimes \delta_\sigma) \times \T$.
In order for the constructions of \cite{andersenmetaplectic} to apply, the family of complex structure parametrised by $\T$ must be rigid as defined in Definition~5.3 in \cite{andersenmetaplectic}. That is, the infinitesimal deformations of the complex structures parametrised by this family are contained in $\omega \cdot H^0(M_\sigma, S^2T_{M_\sigma})$. 

Suppose now $V$ is a vector field on $\T$ and let $\mathbf{G}(V) \in H^0(M_\sigma, S^2T_{M_\sigma})$ be such that the infinitesimal change of the complex structure on $M$ along $V$ is given by $\omega \cdot \mathbf{G}(V)$. Recall by Definition~3.2 in \cite{andersenmetaplectic} that there is a natural connection $\nabla^{r}$ in ${\mathcal H}^{(k)}$ called the reference connection.

\begin{thm}[{\cite[Theorem 1.2]{andersenmetaplectic}}] \label{thm:metahitchinexists}
Under the above assumptions the metaplectic-corrected Hitchin connection 
$$\updel\nabla^{\rm H} \defeq \nabla^r - u$$ 
preserves $\updel H^{(k)} \subset \updel {\mathcal H}^{(k)}$ where
$u \in \Omega^1(\T,\Diff(\L^{\otimes k} \otimes \delta))$ is given by
\begin{align} \label{eq:metahitchindef}
u(V) &\defeq \frac{1}{4k}\left(\Delta_{\mathbf{G}(V)} - \beta(V) \right)
\end{align}
Here $\Delta_{\mathbf{G}(V)} $ is the second order differential operator with symbol $\mathbf{G}(V)$ defined in \cite[Diagram~16]{andersenmetaplectic} and $\beta \in \Omega^1(\T,C^\infty(M))$ is a solution to 
\[
\overline\partial_M\beta(V) = \frac i 2 \Tr \nabla(\mathbf{G}(V)\rho).
\]
guaranteed to exist and be unique up addition of $\Omega^1(\T)$ if $H^0(M_\sigma) = \C$ for all $\sigma \in \T$.

This theorem in particular implies that the metaplectic-corrected Hitchin connection is given by local differential operator, and so it makes sense to consider it restricted to any open subset of $M$ times any (open) submanifold of $\T$.
\end{thm}

We now observe that all the above stated assumptions in this section is satisfied for the moduli space $\mathcal M'_{\vlambda+\vec \rho,k+h} $ with the symplectic form $\omega_{\vlambda+\vec \rho,k+h}$ and the Teichmüller family $\T$ of Kähler structures on it for sufficiently large $m$. Thus we get the following theorem.

\begin{thm} \label{thm:metaexists}
If $m$ is sufficiently large, then the assumptions of \Cref{thm:metahitchinexists} are satisfied and the bundle $\updel H_{\vlambda_m+\vrho,k+h, \sigma}$ is equipped with a metaplectic-corrected Hitchin connection $\updel\nabla^{\rm H}$ which is invariant under the action of the mapping class group $\Gamma$. Moreover this connection is projectively flat.
\end{thm}

\begin{proof}
Since the Knizhnik--Zamolodchikov connection is flat, our main theorem below, Theorem \ref{thm:main}, implies that the Hitchin connection is projectively flat.
\end{proof}

Note that we do not in the proof of our main theorem below use the flatness of the connection constructed in this section.

\section{Projective equivalence of the Hitchin and Knizhnik--Zamolodchikov connections}

This section presents the main theorem of the paper, the projective equivalence of the Hitchin and Knizhnik--Zamolodchikov connection, as well as two important consequences of this theorem. Specifically, the metaplectic-corrected Hitchin connection of the shifted level and weights also defines a Hitchin connection for the original level and weights, and that there is a projectively unique Hitchin connection. In particular, this implies that the heat kernel construction of the Hitchin connection in \cite{biswas2024geometrization} coincides with the definition of the Hitchin connection from \cite{andersen2012hitchin, andersenmetaplectic}.

\begin{thm}[Main Theorem] \label{thm:main}
Assume $n \geq 3$ and let $\vlambda \in \Lambda_k^{\times n}$ such that $\sum_i \lambda_i$ is in the root lattice and assume the moduli space $\MM_{\vlambda,k}$ has smooth points. Then for sufficiently large $m$ the metaplectic-corrected Hitchin connection restricts to a connection in the bundle of covacua $\mathcal V_{\vlambda,k}^\dagger \cong \imag\Psi \subseteq \updel H_{\vlambda_m + \vrho, k + h}$ which is projectively equivalent to the Knizhnik--Zamolodchikov connection on $\mathcal V_{\vlambda,k}^\dagger$.
\end{thm}

In particular, it follows that the metaplectic-corrected Hitchin connection $\updel\nabla^{\rm H}$ by combining with Theorem \ref{thm:pauly} also preserves the subbundles $\tilde\pi (H_{\ell\vlambda,\ell k}) \subseteq \updel H_{\ell\vlambda_m + \vrho,\ell k+h}$ for positive all integers $\ell$ and sufficiently large $m$. This means in particular that we get an induced connection in $H_{\vlambda,k}$ under the assumptions of our main theorem, since $\tilde \pi$ is an isomorphism of bundles onto its image.

The proof of our main theorem will be performed for all $\ell$ simultaneously and the results for different values of $\ell$ will be compared to obtain the result. The strategy will be to show the different orders of the differential operator scale differently with $\ell$ and therefore the second and first order must vanish.

\begin{proof}
If $\mathcal V_{\vlambda,k}^\dagger = 0$ the statement is vacuous, so assume that it does not so that all the results in \Cref{iso} apply. The proof will be a comparison of the two differential operators defining the two connections. To this end they must be viewed on the same space.

For $\vlambda_m = (\vlambda,\vec 0) \in \Lambda_k^{\times (n+m)}$, there are natural isomorphisms $\MM_{\vlambda_m,k} \cong \MM_\vlambda$ and $F_{\vlambda_m} \cong F_\vlambda$ since the trivial weight 0 induces a trivial flag structure $F_0 = G/P_0 = G/G = \siton \ast$, as well as isomorphisms $\mathcal V^\dagger_{\vlambda_m,k} \cong \mathcal V^\dagger_{\vlambda,k}$ and $H^0(\MM_{\vlambda_m,k},\mathscr L_{\vlambda_m,k}) \cong H^0(\MM_{\vlambda,k},\mathscr L_{\vlambda,k})$. By appending zeroes to $\vlambda$, therefore, it may be assumed without loss of generality that $m$ is sufficiently large in the sense of \Cref{thm:metaexists} and $\vlambda = \vlambda_m$ so that $\updel H_{\vlambda + \vrho, k + h}$ carries a metaplectic-corrected Hitchin connection.

Recall the non-empty Zariski opens $U_{\vlambda, k}$ introduced in Diagram~(\ref{diag:stacky}). For any positive integer $\ell$, the open $U_{\ell\vlambda,\ell k} = U_{\vlambda,k}$ because the moduli stability condition depends on the ratio $\vlambda/k$, and $\mathtt L_{\ell\vlambda,\ell k} \cong \mathtt L_{\vlambda,k}^{\otimes \ell} \vb U_{\vlambda,k}$.

The forgetful maps $G/B \vb G/P_{\lambda_i}$ and $G/B \vb \siton \ast$ induce a map $F_{\ell \vlambda + \vrho} \to F_{\vlambda}$, which restricts to a partial map $U_{\ell \vlambda + \vrho, \ell k + h} \dashrightarrow U_{\ell\vlambda,\ell k} = U_{\vlambda,k}$, whose domain is a non-empty Zariski-open, since $U_{\vlambda,k} \subseteq F_\vlambda$ is a non-empty Zariski-open. Without loss of generality, the domain of this partial map may be assumed to be all of $U_{\ell\vlambda + \vrho, \ell k + h}$, as the structure of the proof will only use that it defines a non-empty Zariski-open of both the GIT quotient and the moduli space.

The sequence of weight vectors $\{(\ell\vlambda + \vrho)/(\ell k + h)\}_{\ell \in \N}$ converges to $\vlambda/k$ and is thus bounded. As such, parabolic reductions of sufficiently high degree are never destabilising for any $\ell$. Therefore a finite number of open conditions which determine bundle semistability for each $\ell$ separately. As the semistability conditions are linear in the weights, for every possible combination of weights inherited by any reduction, each condition is either eventually satisfied or eventually violated for sufficiently large $\ell$. Since there are only finitely many degrees for reductions to be considered, and only finitely many possible combinations of weights to be inherited, there is an $\ell_0$ such that any quasi-parabolic structure on the trivial bundle equipped with weights $(\ell\vlambda + \vrho)/(\ell k + h)$ is either semistable for all $\ell \geq \ell_0$ or unstable for all $\ell \geq \ell_0$.

Hence, we may choose $U_{\ell\vlambda + \vrho, \ell k + h} = U_{\ell_0 \vlambda + \vrho, \ell_0 k + h}$ for all $\ell \geq \ell_0$. Write therefore $U \defeq \bigcap_\ell U_{\ell \vlambda + \vrho, \ell k + h}$ for the finite intersection of non-empty Zariski-opens, with symplectic form $\omega_{\vlambda,k}$ that is the curvature of the restriction of $(\mathtt L_{\vlambda,k},\nabla)$.

Finally, \Cref{prop:2ODO} gives a presentation of the Knizhnik--Zamolodchikov connection in terms of covariant derivatives with respect to vector fields defined explicitly in terms of $\g$-action. Since the map $G/B \vb G/P_\lambda$ is $G$-equivariant, these vector fields can canonically be lifted from $U_{\vlambda,k}$ to $U$. Following the observation in the proof of \Cref{thm:meta-incl}, the line bundle $L_{\vlambda,k}^{\otimes \ell} = L_{\ell \vlambda,\ell k} \vb U$ coincides with $L_{\vlambda+\vrho,\ell k + h} \otimes \delta \vb U$, where $\delta$ is the metaplectic bundle, which exists by \Cref{lem:evenchern}, since we may assume $n$ even without loss of generality by appending a zero to the weight vector if needed, just like before. Therefore the differential operator $\bf\Omega$ defining $\nabla^{\rm KZ}$ on sections of $\mathtt L_{\vlambda,k}^{\otimes \ell} \vb U_{\vlambda,k}$ in terms of covariant derivatives along these vector fields with respect to the Bott--Borel--Weil connection $\nabla$ canonically pulls back to a differential operator on the pullback of $(L_{\vlambda,k}^{\otimes \ell},\nabla)$ to $U$. Thus the Knizhnik--Zamolodchikov connection and the metaplectic-corrected Hitchin connection are both defined in terms of second-order differential operators on $\mathcal L_{\vlambda, k}^{\otimes \ell} \vb U$ and their difference may now be considered.

For a positive integer $\ell$ and a tangent vector $V \in T_\sigma \T$, write the difference $\updel\nabla^{\rm H}_V - \nabla^{\rm KZ}_V$ as
\[
u_\ell(V) \defeq \updel\nabla^{\rm H}_V - \nabla^{\rm KZ}_V \eqdef \nabla^2_{\mathcal G_\ell(V)} + D_\ell(V) \in \Diff^{\leq 2}(U,\mathtt L_{\ell \vlambda, \ell k}) \cong \Diff^{\leq 2}(U,\mathtt L_{\vlambda,k}^\ell)
\]
where $D_\ell(V)$ is a first-order differential operator, $\mathcal G_\ell(V) \in C^\infty (S^2T_U)$ is the second-order symbol, and $\nabla^2$ is given by $\nabla^2_{X \otimes Y} = \nabla_X \nabla_Y - \nabla_{\nabla^{LC}_X Y}$ where $\nabla^{LC}$ is Levi-Civita connection with respect to the Kähler metric.

By construction, the metaplectic-corrected Hitchin connection preserves the space holomorphic sections of $\mathscr L_{\vlambda,k,\sigma}^{\otimes \ell}$ over any open subset of  $\MM'_{\ell \vlambda + \vrho,\ell k + h,\sigma}$ and the Knizhnik--Zamolodchikov connection preserves specifically the image of $H^0(\FF_{\vlambda,k},\L_{\vlambda,k})$ under $\Phi_\sigma$ inside $H^0(\MM'_{\ell \vlambda + \vrho,\ell k + h,\sigma},\mathscr L_{\vlambda,k,\sigma}^{\otimes \ell})$ as in \Cref{lem:quantiso}. Therefore it follows from \cite[Lemma 5.1]{andersenmetaplectic} that $\nabla^{0,1}u_\ell(V) s = 0$ whenever $s \in H^0(U,\mathtt L_{\vlambda,k}^{\otimes \ell})$ is the restriction of an element in $\imag\Phi_\sigma = \imag\Psi_\sigma \subseteq H^0(\MM_{\ell \vlambda + \vrho,\ell k + h,\sigma},\mathscr L_{\vlambda,k,\sigma}^{\otimes \ell})$.

This is equivalent to the commutator $[\nabla^{0,1},u_\ell(V)]$ vanishing on $\imag \Psi_\sigma$. Together with the observation that $u_\ell(V)[\imag \Psi_\sigma] \subseteq H^0(\MM_{\ell\vlambda + \vrho,\ell k + h,\sigma},\L_{\vlambda,k,\sigma}^{\otimes \ell})$ this will imply the theorem by evaluating the $\ell$-scaling behaviour.

Without loss of generality, $\mathcal G_\ell(V)$ may be assumed to be of purely holomorphic type, since the difference is seen to be a first-order differential operator, as for any holomorphic section $s$ and vector fields $X,Y$
\[
(\nabla_X \nabla_Y - \nabla_{X^{1,0}}\nabla_{Y^{1,0}})s &= (\nabla_X \nabla_Y - \nabla_{X}\nabla_{Y^{1,0}} + \nabla_{X} \nabla_{Y^{1,0}} - \nabla_{X^{1,0}}\nabla_{Y^{1,0}})s\\
&= (\nabla_X \nabla_{Y^{0,1}} + \nabla_{X^{0,1}}\nabla_{Y^{1,0}})s = 0 + (\nabla_{[X^{0,1},Y^{1,0}]} + \ell\omega_{\vlambda,k}(X^{0,1},Y^{1,0}))s
\]
using that $\ell\omega_{\vlambda,k}$ is the curvature of $\nabla$ on $\mathtt L^\ell_{\vlambda,k}$ and
\[
2\nabla_X\nabla_Y - \{\nabla_X,\nabla_Y\} &= [\nabla_X,\nabla_Y] = \nabla_{[X,Y]} + \ell\omega(X,Y)
\]
where the anti-symmetriser is $\{\nabla_X,\nabla_Y\} = \nabla_X\nabla_Y + \nabla_Y\nabla_X$. As such, without loss of generality, this first-order differential operator can be taken as part of $D_\ell(V)$, so it may be assumed that $\mathcal G_\ell(V) \in C^\infty (S^{2,0}(TU))$.

Working out the bracket $[\nabla_W,\nabla^2_{\mathcal G_\ell}]$ gives
\[
 [\nabla_W,\{\nabla_X,\nabla_Y\}] &= \{[\nabla_W,\nabla_X],\nabla_Y\} + \{\nabla_X,[\nabla_W,\nabla_Y]\} \\
 &= \{\nabla_{[W,X]} + \ell\omega_{\vlambda,k}(W,X),\nabla_Y\} + \{\nabla_X,\nabla_{[W,Y]} + \ell\omega_{\vlambda,k}(W,Y)\}\\
&= \{\nabla_{[W,X]},\nabla_Y\} + \{\nabla_X,\nabla_{[W,Y]}\} + 2\ell(\omega_{\vlambda,k}(W,X)\nabla_Y + \omega_{\vlambda,k}(W,Y)\nabla_X) \\
&+ \ell(\omega_{\vlambda,k}(\nabla^{LC}_X W, Y) + \omega_{\vlambda,k}(\nabla^{LC}_Y W,X) + \omega_{\vlambda,k}(W,\nabla^{LC}_X Y + \nabla^{LC}_Y X)),
\]
where $\nabla^{LC}$ is the Levi-Civita connection, and
\[
[\nabla_W,\nabla_{\nabla_X^{LC} Y}] &= \ell\omega_{\vlambda,k}(W,\nabla_X^{LC} Y) + \nabla_{[W,\nabla_X^{LC} Y]} \\
{[W,\nabla_X^{LC} Y]} &= {(\Lie_W \nabla^{LC})_X Y} + {\nabla^{LC}_{[W,X]} Y} + {\nabla^{LC}_X [W,Y]}
\]
so 
\begin{equation} \label{eq:comm}
[\nabla_W,\nabla^2_{\mathcal G_\ell(V)}] = \nabla^2_{[W,\mathcal G_\ell(V)]} + \ell(2\omega_{\vlambda,k}(W) \otimes \nabla_- + \omega_{\vlambda,k}(\nabla_-^{LC}W))\mathcal G_\ell(V) - \nabla_{(\Lie_W \nabla^{LC}_-) \mathcal G_\ell(V)}
\end{equation}
with 
\[
[W,X \otimes Y] &\defeq [W,X]\otimes Y + X \otimes [W,Y]; \\
(\omega_{\vlambda,k}(W) \otimes \nabla_-)(X \otimes Y) &\defeq \omega_{\vlambda,k}(W,X)\nabla_Y; \\ 
\omega_{\vlambda,k}(\nabla_-^{LC}W)(X \otimes Y)&\defeq \omega_{\vlambda,k}(\nabla_X^{LC} W,Y);\\
(\nabla_-^{LC})(X \otimes Y) &\defeq \nabla_X^{LC}Y.
\]
By construction of the Hitchin and the Knizhnik--Zamolodchikov connections in Equations~\ref{eq:metahitchindef} and \ref{eq:KZ}, $u_\ell$, $\mathcal G_\ell$, and $D_\ell$ scale with $\ell$ as
\[
u_\ell &= \frac{k + h}{\ell k + h}u_1 & \mathcal G_\ell &= \frac{k + h}{\ell k + h}\mathcal G_1 & D_\ell &= \frac{k + h}{\ell k + h}D_1. &
\]

Since \Cref{eq:comm} is valid for all $\ell$, each component with different scaling must be 0 independently so it follows that $(2\omega_{\vlambda,k}(W) \otimes \nabla_- + \omega_{\vlambda,k}(\nabla_-^{LC}W))\mathcal G_\ell(V) = 0$ for all smooth $W = W^{0,1}$. At any point $x$ a smooth function $f$ with $f(x)=0$ reduces this equation to $(\omega_{\vlambda,k}(W,-) \otimes \d f)\mathcal G_\ell(V) = 0$. Since this holds for all $f$ and all antiholomorphic vector fields, it follows that $\mathcal G_\ell(V) = 0$ by the non-degeneracy of the type-(1,1) form $\omega_{\vlambda,k}$.

Therefore $u_\ell(V)s|_U = D_\ell(V)s|_U$ for any $s \in \imag\Psi_\sigma$. Writing $D_\ell(V) = \nabla_{Z_\ell(V)} + f_\ell(V)$, without loss of generality $Z_\ell(V) = Z_\ell^{1,0}(V) \in C^\infty(T^{1,0}U_{\vlambda,k})$, because anti-holomorphic vector fields act trivially on holomorphic sections. Then
\[
[\nabla_W,D_\ell(V)] &= [\nabla_W,\nabla_{Z_\ell(V)}] + W(f_\ell(V)) = \nabla_{[W,Z_\ell(V)]} + \ell\omega_{\vlambda,k}(W,Z_\ell(V)) + W(f_\ell).
\]

By the same scaling argument $\omega_{\vlambda,k}(W,Z_\ell) = 0$ for all $\ell$, so that the non-degeneracy of $\omega_{\vlambda,k}$ of type (1,1) again implies $Z_\ell(V) = 0$ and therefore $W(f_\ell(V)) = 0$.

Therefore $u_\ell(V)s|_U = f_\ell(V) \cdot s|_U$ for $f_\ell(V) \in H^0(U)$ whenever $s \in \imag \Psi_\sigma$, which defines a meromorphic function $\tilde f_\ell(V)$ on $\MM_{\ell \vlambda + \vrho,\ell k + h}$. 

Assume for the sake of contradiction that $\tilde f_\ell(V)$ has a pole of order $n > 0$. Then, considering that $\nabla^{\rm KZ}_V$ preserves $\imag \Psi_\sigma \subseteq H^0(\MM_{\ell\vlambda + \vrho,\ell k + h},\mathscr L_{\vlambda,k}^{\otimes \ell})$, it follows that for any $0 \neq s \in \imag\Psi_\sigma$ there is an integer $m$ such that
\[
(\nabla^{\rm KZ}_V)^m s = (\updel \nabla^{\rm H}_V - \tilde f_\ell(V))^m s
\]
is not holomorphic, in contradiction with the fact that $\nabla^{\rm KZ}_V$ and $\updel \nabla^{\rm H}_V$ preserve the holomorphicity of $s \in \imag\Psi_\sigma$. Therefore $\tilde f_\ell(V)$ is holomorphic and since $\MM_{\ell\vlambda + \vrho,\ell k + h}$ is compact, it is constant. This concludes the proof.
\end{proof}

Combining \Cref{thm:main} and \Cref{thm:meta-incl} immediately gives the following.

\begin{corollary}
The metaplectic-corrected Hitchin connection $\updel\nabla^{\rm H}$ on $\updel H_{\vlambda + \vrho, k + h}$ for the moduli space $\MM_{\vlambda + \vrho,k + h}$ pulls back along the inclusion $\tilde\pi$ from \Cref{thm:meta-incl} to a Hitchin connection on the Verlinde bundle $H_{\vlambda,k}$ for the moduli space $\MM_{\vlambda,k}$.
\end{corollary}

Moreover, the proof strategy above can be applied to state an important uniqueness theorem for Hitchin connections on Verlinde bundles.

\begin{corollary} \label{cor:unique}
For any $\vlambda \in \Lambda_k^{\times n}$, where $n \geq 3$, any two Hitchin connections in the Verlinde bundle $H_{\vlambda,k}$ are projectively equivalent.
\end{corollary}

\end{document}